\documentclass[12pt]{amsart} 

\usepackage{amsthm,amssymb}
\usepackage[english]{babel}
\usepackage{fullpage}
\usepackage{amsmath}
\usepackage{amsfonts}
\usepackage{graphicx}
\usepackage[colorlinks=true, allcolors=blue]{hyperref}
\usepackage{amssymb}
\usepackage{xcolor}
\usepackage{float}
\newcommand{\stw}{\mathbb{S}^2}

\usepackage{soul,xcolor}

\newtheorem{theorem}{Theorem}
\newtheorem{lemma}{Lemma}

\newtheorem{rem}{Remark}
\newtheorem{question}{Question}

\begin{document}

\title{The 10 antipodal pairings of strongly involutive polyhedra on the sphere}

\author[Javier Bracho]{Javier Bracho}
\address{Instituto de Matem\'aticas, UNAM, Ciudad Universitaria, M\'exico, CP. 04510}
\email{jbracho@im.unam.mx}
\author[Luis Montejano]{Luis Montejano}
\address{Instituto de Matem\'aticas, UNAM, Campos Juriquilla,
Quer\'etaro, M\'exico, CP. 07360}
\email{luis@im.unam.mx}
\author[Eric Pauli]{Eric Pauli}
\address{Instituto de Matem\'aticas, UNAM, Campos Juriquilla,
Quer\'etaro, M\'exico, CP. 07360}
\email{eric@im.unam.mx}
\author[Jorge L. Ram\'irez Alfons\'in]{Jorge L. Ram\'irez Alfons\'in}
\address{IMAG, Univ.\ Montpellier, CNRS, Montpellier, France }
\email{jorge.ramirez-alfonsin@umontpellier.fr}


\maketitle

\begin{abstract}

It is known that {\em strongly involutive} polyhedra are closely related to self-dual maps where the antipodal function acts as duality isomorphism. Such a family of polyhedra appears in different combinatorial, topological and geometric contexts, and is thus attractive to be studied. In this note, we determine the 10 {\em antipodal pairings} among the classification of the 24 self-dual pairings $Dual(G)\rhd  Aut(G)$ of self-dual maps $G$. We also present the {\em orbifold} associated to each antipodal pairing and describe explicitly the corresponding {\em fundamental regions}.
We finally explain how to construct two infinite families of strongly involutive polyhedra (one of them new) by using their {\em doodles} and the action of the corresponding orbifolds. 
\end{abstract}

\section{Introduction}

A {\em polyhedron} is a planar graph $G$ that is simple (without loops and multiple edges) and 3-connected. A {\em face} $F$ of an embedding of a planar graph is a region bounded by a cycle $C$. We say that a vertex $v$ belongs to a face $F$, denoted by $v\in F$, if $v$ is in the cycle $C$.  Note that every edge borders exactly two faces. The {\em dual} of $G$, denoted by $G^*$, is defined as follows:  Each face in $G$ is a vertex in $G^*$ and two vertices in $G^*$ are adjacent if and only if the faces share an edge in $G$. A polyhedron $G$ is said to be {\em self-dual} if there is an isomorphism of graphs $\tau : G \rightarrow G^{*}$. This isomorphism is called a \emph{duality isomorphism}, and it will be called a \textit{strong involution} if it satisfies the following conditions:
\smallskip

$(i)$ For each pair $u,v$ of vertices, $u\in \tau (v)$ if and only if $v\in \tau (u)$ and

$(ii)$ For every vertex $v$, we have that $v\notin \tau (v)$.
\smallskip

We say that a self-dual polyhedron is \textit{strongly involutive} if it admits a strong involution. The above conditions are the combinatorial counterpart (in the 3-dimensional case) of a more general geometric object called {\em strongly self-dual polytopes}, introduced by Lov\'asz \cite{L}, see also \cite{G}. Strongly involutive polyhedra are relevant in different contexts: they are related to convex geometric problems such as the well-known V\'azsonyi's problem \cite{E,G1}, with ball polyhedra \cite{BLNP,KMP,KS}, the chromatic number of distance graphs on the sphere \cite{L}, the Reuleaux polyhedra \cite{gyivan, reuleaux}, constant width bodies \cite{MMO} and more recently in relation with questions concerning the symmetry as well as the amphicheirality of knots \cite{MRAR2} and projective links \cite{MRAR3}.

There are known infinite families of strongly involutive polyhedra \cite[Propositions 1, 2 and 3 and Theorem 3]{MRAR1}. Figure \ref{fig1} illustrates two simple examples.

\begin{figure}[h]
\centering
\includegraphics[scale=0.2]{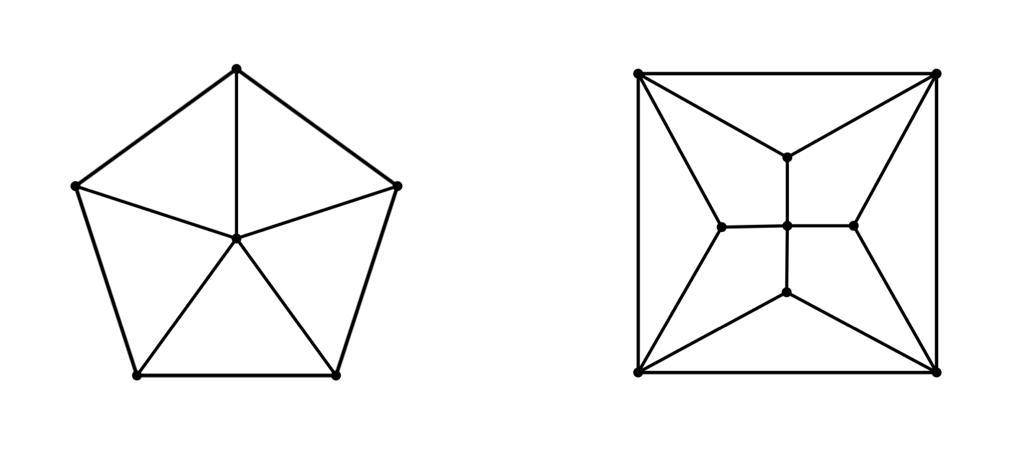}
\caption{(Left) 5-wheel. (Right) 4-hyperwheel.}
\label{fig1}
\end{figure}

Strongly involutive isomorphisms are closely related to the notion of {\em antipodal isometries} in the 2-sphere. 

A \textit{map} $G=(V, E, F)$ is the image of an embedding of $G$ into $\mathbb{S}^2$ where the set of vertices is a collection of distinct points in $\mathbb{S}^2$ and the set of edges is a collection of arcs joining pairs of points in $V$ satisfying that for any pair of arcs, say $a$ and $a'$, we have that $a \cap a'$ is either empty or an endpoint of one of the arcs. Any embedding of the topological realization of $G$ into $\mathbb{S}^2$ partitions the 2-sphere into simply connected regions of $\mathbb{S}^2\setminus G$ called the \textit{faces} $F$ of the embedding.
\smallskip

Recall that $\mbox{Aut}(G)$ is the group formed by the set of all {\em automorphisms} of $G$ (i.e., the set  of isomorphisms of $G$ into itself) and let $\mbox{Iso}(G)$ be the set of all duality isomorphisms of $G$ into $G^*$. We notice that $\mbox{Iso}(G)$ is not a group since the composition of any two of them is an automorphism. Let us suppose that $G=(V,E,F)$ is a self-dual map so that there is an isomorphism $\phi:(V,E,F)\rightarrow (F^*,E^*,V^*)$. Following $\phi$ with the correspondence $*$ gives a permutation on $V\cup E\cup F$ which preserve incidences but reverses dimension of the elements. The collection of all such permutations generates a group $\mbox{Dual}(G)=\mbox{Aut}(G)\cup \mbox{Iso}(G)$ in which the automorphism group $\mbox{Aut}(G)$ is contained as a subgroup of index 2.   
 \smallskip
 
It is known \cite[Lemma 1]{SS} that for a given map $G$ there is a homeomorphism $\rho$ of $\mathbb{S}^2$ to itself such that for every $\sigma\in \mbox{Aut}(G)$ we have that $\rho\sigma\in \mbox{Isom}(\mathbb{S}^2)$ where $\mbox{Isom}(\mathbb{S}^2)$ is the group of isometries of the 2-sphere. In other words, any planar graph $G$ can be drawn on the 2-sphere such that any automorphism of the resulting map acts as an isometry of the sphere.  This was extended in \cite{SS} by showing that given any self-dual planar graph $G$ and its dual $G^*$ can be drawn on the 2-sphere so that $\mbox{Dual}(G)$ is realized as a group of spherical isometries. 
\smallskip

In \cite[Theorem 9]{BMPRA} (see also \cite[Corollary 2]{MRAR1}), it was proven that if $G$ is strongly involutive then the corresponding isometry is the antipodal function ${\alpha : \mathbb{S}^2 \rightarrow \mathbb{S}^2}$, $\alpha (x) = -x$.


The couple $\mbox{Dual}(G) \rhd \mbox{Aut}(G)$ is called the {\em self-dual pairing} of the map $G$. In \cite{SS}, all self-dual maps were enumerated and classified. In the notation of \cite[page 36-38]{coxeter} the 24 possible pairings are:
 \smallskip
 
 - among the infinite classes $[2,q]\rhd [q], [2,q]^+\rhd [q]^+, [2^+,2q]\rhd [2q], [2,q^+]\rhd [q]^+$ and $[2^+,2q^+]\rhd [2q]^+$ with $q$ positive integer, or 
 \smallskip
 
 - among the special pairings $[2]\rhd [1],  [2]\rhd [2]^+, [4]\rhd [2], [2]^+\rhd [1]^+, [4]^+\rhd [2]^+, [2,2]\rhd [2,2]^+, [2,4]\rhd [2^+,4], [2,2]\rhd [2,2^+], [2,4]\rhd [2,2], [2,4]^+\rhd [2,2]^+, [2^+,4]\rhd [2,2]^+, [2^+,4]\rhd [2^+,4^+], [2,4^+]\rhd [2^+,4^+], [2,2^+]\rhd [2^+,2^+], [2,4^+]\rhd [2,2^+], [2,2^+]\rhd [1], [3,4]\rhd [3,3], [3,4]^+\rhd [3,3]^+$ and $[3^+,4]\rhd [3,3]^+$.
\smallskip

In Section \ref{sec2}, we classify the self-dual pairings $\mbox{Dual}(G) \rhd \mbox{Aut}(G)$ that correspond to antipodal isometries, that is, self-dual pairings that give rise to strongly involutive polyhedra.

\begin{theorem}\label{thm1} Every strongly involutive map corresponds to one of the following self-dual pairings: $[q]\triangleleft [2,q],[q]^{+}\triangleleft [2,q^{+}]$  for $q$ even; ${[q]\triangleleft [2^{+},2q]}, {[q]^{+}\triangleleft [2^{+},2q^{+}]}$  for $q$ odd, $[2,2]^{+}\triangleleft [2,2], [2^{+},4]\triangleleft [2,4], [2^{+},4^{+}]\triangleleft [2,4^{+}], [1]\triangleleft [2,2^{+}],[3,3]\triangleleft [3,4] \mbox{ or } [3,3]^{+}\triangleleft [3^{+},4]$.
\end{theorem}

The self-dual pairings given in Theorem \ref{thm1} will be called {\em antipodal pairings}. We shall also present a short description of the {\em orbifold} associated to each antipodal pairing. We do so by explicitly giving the {\em fundamental region} for each group and its generators. 
\smallskip

In Section \ref{sec:construction}, we study the {\em doodle} (the map restricted to its fundamental region) of two special families of strongly involutive polyhedra: the {\em $q$-multi hyperwheels} (a new family of strongly involutive polyhedra) and the {\em $q$-multi wheels}.
We explain how these families can be constructed from their doodle via the action of the corresponding orbifold. We finally end with some concluding remarks.

\section{Antipodal pairings}\label{sec2}

\subsection{Spherical isometries} We quickly recall some notions on the group of spherical isometries (we refer the reader to \cite{coxeter} for further details and where our notation is taken from). An {\em isometry} of $\mathbb{S}^2$ is either a rotation or a reflection or a composition of these two. Let $\mbox{Isom}(\stw)$ be the group of isometries of the 2-sphere. We remark that any orientation re\-ver\-sing isometry is either a reflection or a product of three reflections. Also, it can be checked that $\mbox{Isom}(\stw)$ is generated by only reflections. Moreover, any isometry is a product of at most three reflections:
\smallskip

$\bullet$ [one reflection] Let $H$ be a plane passing through the center of $\stw$. The reflection $\gamma$ w.r.t.\ $H$ is an involutive isometry and therefore of order 2. This group is denoted by $[1]$.
\smallskip

$\bullet$ [two reflections] Let $H_1$ and $H_2$ be two planes with angle $\angle (H_1,H_2)=\frac{\pi}{q}$. The reflection $\gamma_1$ and $\gamma_2$ w.r.t. these planes generate the group $[q]$ (isomorphic to the dihedral group $D_q$) having presentation $\gamma_1^2=\gamma_2^2=(\gamma_1\gamma_2)^q=id$.
\smallskip

$\bullet$ [three reflections] Let $H_1, H_2$ and $H_3$ be three planes inducing a spheric triangle with angles $\frac{\pi}{p}, \frac{\pi}{q}$ and $\frac{\pi}{r}$, denoted by $[p,q]$. 

\begin{figure}[H]
\centering
\includegraphics[width=0.7\linewidth]{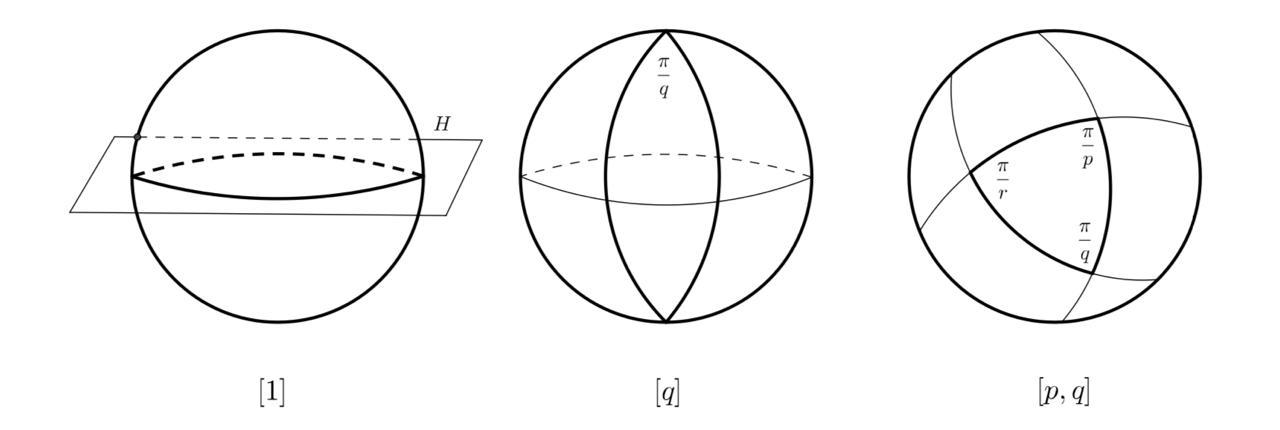}
\caption{Reflection groups.}
\label{fig22}
\end{figure}

It is known that the sum of the angles of a spheric triangle is strictly larger that $\pi$, therefore $\frac{1}{p}+\frac{1}{q}+\frac{1}{r}>1$ implying that $\min\{p,q,r\}=2$. We take $r=2$, obtaining that $\frac{1}{p}+\frac{1}{q}>\frac{1}{2}$
yielding to $(p-2)(q-2)<4$. By analyzing the latter, it can be found that the possible cases are : 

- $[3,3]$ (corresponding to the tetrahedron),

- $[3,4]$ (corresponding to the cube and octahedron),  

- $[3,5]$ (corresponding to the dodecahedron and icosahedron),

- $[2,q]$ (corresponding to a {\em diamond} polyhedron which dual is a {\em prism} with a $q$-polygon as a base).
\smallskip

The group $[p,q]$ or $[q,p]$ is defined by $\gamma_1^2=\gamma_2^2=\gamma_3^2=(\gamma_1\gamma_2)^p=(\gamma_2\gamma_3)^q=(\gamma_1\gamma_3)^2=id$ having as a {\em fundamental region} the spherical triangle with angles $\frac{\pi}{p},\frac{\pi}{q}$ and $\frac{\pi}{2}$. Each of these groups has a rotational group consisting only of the possible rotations generated by pairs of reflections:
\smallskip

- The rotational subgroup of $[q]$, consisting of the powers of $\gamma_1\gamma_2$, is denoted by $[q]^+$ (clearly $[1]^+$ is the trivial group),

- the rotational subgroup of $[p,q]$, generated by the rotations  $\gamma_1\gamma_2, \gamma_2\gamma_3$ and $\gamma_3\gamma_1$, is denoted by $[p,q]^+$ (these subgroups are of order 2), 

- the rotational subgroup of $[p,q]$ when $q$ is even, generated by $\gamma_1\gamma_2$ which is a rotation of order $p$ followed by the reflection $\gamma_3$, is denoted by $[p^+,q]$ or $[q,p^+]$ (these subgroups are of order 2),

- the rotational subgroup of $[2^+,q]$, generated by $\gamma_1\gamma_2\gamma_3$, is denoted by $[2^+,q^+]$ (which turns out to be a cyclic group of index 2)  generated by the rotatory reflection $\rho \gamma _3$.

\subsection{Detecting antipodal pairings} We first present the following lemma that identifies which groups contain the antipodal function.

\begin{lemma}\label{lem1} Let $\alpha :\mathbb{S}^2\rightarrow\mathbb{S}^2$ be the antipodal function. Then,
\begin{enumerate}
\item $\alpha \notin [q]$ for any positive integer $q$.
\item $\alpha \notin [q]^{+}$ for any positive integer $q$.
\item $\alpha \in [p,q]$ if and only if $p=3,q=4$ or $p=2$ and $q$ is even.
\item $\alpha \notin [p,q]^{+}$ for any positive integers $p$ and $q$.
\item $\alpha \in [p^{+},q]$ or $[q,p^{+}]$ if and only if $p=3,q=4$ or $p=2$, $q$ is even and $\frac{q}{2}$ is odd.
\item $\alpha \in [p^{+},q^{+}]$ if and only if $p=2$, $q$ is even and $\frac{q}{2}$ is odd.
\end{enumerate}
\end{lemma}

\begin{proof}
Since $\alpha$ is an isometry without fixed points then (1) is verified. Moreover, since $\alpha$ inverts orientation then (2) and (4) hold. Notice that (2) also follows since $[q]^+$ is a subgroup of $[q]$.
We can establish (3) by observing that any prism has central symmetry if and only if it has a polygon with an even number of sides as a base and similarly we see that the cube and the octahedron are the only platonic solids on the list with central symmetry.

In order to illustrate the proof of (5) and (6), let us see that $\alpha \in [2^{+},6]$. The group $[2,6]$ is generated by the reflections $\gamma _1,\gamma_2$ and $\gamma_3$ such that $\gamma_1^2=\gamma_2^2=\gamma_3^2=(\gamma_1\gamma_2)^2=(\gamma_3\gamma_1)^2=(\gamma_2\gamma_3)^2=id$. The subgroup $[2^{+},6]$ is generated by the reflection $\gamma=\gamma_1$ and the rotation $\rho=\gamma_2\gamma_3$. The element $\rho\gamma$ is a \textit{rotatory reflection} or a \textit{step}, satisfying $(\rho\gamma)^6=id$ and $(\rho\gamma)^{3}=\alpha$ (See Figure \ref{zigzag}).
\begin{figure}[h]
\centering
\includegraphics[scale=0.3]{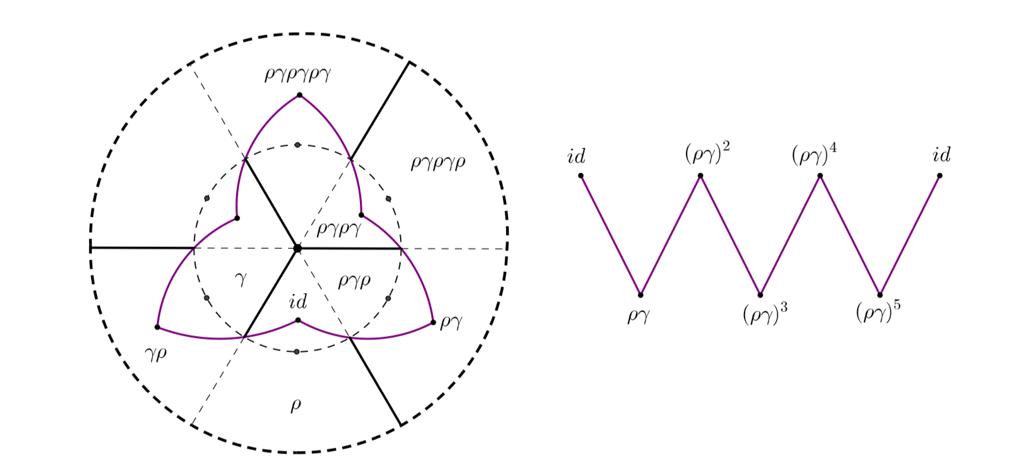}
\caption{Zig-zag described by $(\rho\gamma)^i$ in $[2^{+},6]$. The exterior dotted circle represents one of the poles, the central black circle the other pole and the middle dotted circle the equator.}
\label{zigzag}
\end{figure}

In the general case $[2^{+},q]$, if $q$ is even and $\frac{q}{2}$ is odd, then we have that $(\rho\gamma)^q=id$ and $(\rho\gamma)^{\frac{q}{2}}=\alpha$. Conversely, if $(\rho\gamma)^{\frac{q}{2}}=\alpha$, then $q$ is even and $\frac{q}{2}$ must be odd in order to express $\alpha$ as a product of generators with an odd number of reflections, otherwise $\alpha$ would preserve orientation, which is not the case. From this fact we can conclude also 6 since $[2^{+},q^{+}]$ is precisely the cyclic group generated by the step $\rho\gamma$. Finally, by considering (3), we just need to check the subgroup $[3^{+},4]$, which is generated by a reflection $\gamma$ and a rotation $\rho$. In this case, it can be verified that $\alpha=(\rho\gamma)^3$.
\end{proof}
\smallskip

We may now prove Theorem \ref{thm1}.
\smallskip

{\em Proof of Theorem \ref{thm1}.} We may use Lemma \ref{lem1} to identify those self-dual pairings $\mbox{Dual}(G) \rhd \mbox{Aut}(G)$ such that $\alpha \in \mbox{Dual}(G)$ and $\alpha \notin \mbox{Aut}(G)$. In Table \ref{tab1} we examine each case. \hfill $\square$

\begin{table}[H]
\begin{center}
\begin{tabular}{|l|c|c|}
\hline
\text{Dual$(G)\rhd$Aut$(G)$} & $\alpha\in$ Dual$(G$)? & $\alpha\in $ Aut$(G)$?\\
\hline
\hline
$\mathbf{[2,q]\rhd [q]}$, $q$-even & $\alpha\in[2,q]$ (3) &  $\alpha\not\in[q]$ (1)\\
\hline
$[2,q]^+\rhd [q]^+$ & $\alpha\not\in [2,q]^+$ (4) & $\alpha\not\in[q]^+$ (2)\\
\hline 
$\mathbf{[2^+,2q]\rhd [2q]}$, $q$-odd & $\alpha\in[2^+,2q]$ (5) & $\alpha\not\in[2q]$ (1) \\
\hline
$\mathbf{[2,q^+]\rhd [q]^+}$, $q$-even & $\alpha\in[2,q^+]$ (5)\ & $\alpha\not\in[q]^+$ (2)\\
\hline
$\mathbf{[2^+,2q^+]\rhd [2q]^+}$, $q$-odd & $\alpha\in[2^+,2q^+]$ (6) & $\alpha\not\in[2q]^+$ (2) \\
\hline
 $[2]\rhd [1]$ & $\alpha\not\in[2]$ (1) & $\alpha\not\in[1]$ (1) \\
\hline 
 $[2]\rhd [2]^+$ & $\alpha\not\in[2]$ (1) & $\alpha\not\in[2]^+$ (2) \\
\hline 
 $[4]\rhd [2]$ & $\alpha\not\in[4]$ (1)  & $\alpha\not\in[2]$ (1) \\
\hline
 $[2]^+\rhd [1]^+$ & $\alpha\not\in[2]^+$ (2) & $\alpha\not\in[1]^+$ (1)\\
 \hline
 $[4]^+\rhd [2]^+$ & $\alpha\not\in[4]^+$ (2) & $\alpha\not\in[2]^+$ (2) \\
 \hline
 $\mathbf{[2,2]\rhd [2,2]^+}$ & $\alpha\in[2,2]$ (3) & $\alpha\not\in[2,2]^+$ (4) \\
\hline
 $\mathbf{[2,4]\rhd [2^+,4]}$ &  $\alpha\in[2,4]$ (3) & $\alpha\not\in[2^+,4]$ (5)\\
\hline
 $[2,2]\rhd [2,2^+]$ &  $\alpha\in[2,2]$ (3) & $\alpha\in[2,2^+]$ (5) \\
 \hline
 $[2,4]\rhd [2,2]$ & $\alpha\in[2,4]$ (3) & $\alpha\in[2,2]$ (3) \\
 \hline
 $[2,4]^+\rhd [2,2]^+$ &  $\alpha\not\in[2,4]^+$ (4) & $\alpha\not\in[2,2]^+$ (4) \\
 \hline
 $[2^+,4]\rhd [2,2]^+$ &  $\alpha\not\in[2^+,4]$ (5) & $\alpha\not\in[2,2]^+$ (4) \\
\hline
 $[2^+,4]\rhd [2^+,4^+]$ & $\alpha\not\in[2^+,4]$ (5) & $\alpha\not\in[2^+,4^+]$ (6)\\
 \hline
 $\mathbf{[2,4^+]\rhd [2^+,4^+]}$ & $\alpha\in[2,4^+]$ (5) & $\alpha\not\in[2^+,4^+]$ (6) \\
 \hline
 $[2,2^+]\rhd [2^+,2^+]$ & $\alpha\in[2,2^+]$  (5) & $\alpha\in[2^+,2^+]$ (6) \\
 \hline
 $[2,4^+]\rhd [2,2^+]$ & $\alpha\in[2,4^+]$ (5) & $\alpha\in[2,2^+]$ (5) \\
 \hline
 $\mathbf{[2,2^+]\rhd [1]}$ & $\alpha\in[2,2^+]$ (5) & $\alpha\not\in[1]$ (1)\\
 \hline
 $\mathbf{[3,4]\rhd [3,3]}$ &  $\alpha\in[3,4]$ (3) & $\alpha\not\in[3,3]$ (3)\\
 \hline
 $[3,4]^+\rhd [3,3]^+$ & $\alpha\not\in[3,4]^+$ (4) & $\alpha\not\in[3,3]^+$ (4)\\
 \hline
$\mathbf{[3^+,4]\rhd [3,3]^+}$ & $\alpha\in[3^+,4]$ (5) & $\alpha\not\in[3,3]^+$ (4)\\
 \hline
\end{tabular}
\end{center}
\caption{Detecting if the antipodal map $\alpha$ belongs to $\mbox{Aut}(G)$ and/or $\mbox{Dual}(G)$. In parenthesis the condition of Lemma \ref{lem1} that is applied. In bold the self-dual pairings with $\alpha \in \mbox{Dual}(G)$ and $\alpha \notin \mbox{Aut}(G)$.}\label{tab1}
\end{table}

\subsection{Orbifolds} Given $\sigma\in\mbox{Dual}(G)$ and a point $p \in\stw$ we have that $\sigma\cdot p=\sigma(p)$,
that is, $\mbox{Dual}(G)$ acts as an evaluation function in $\stw$. 
The orbit of $R$  is the set of regions congruent to $R$ covering the sphere (the number of such regions is as many as the elements in $\mbox{Dual}(G)$. The {\em stabilizer} of $p$ consist of all the elements in $\mbox{Dual}(G)$ fixing $p$. An {\em orbifold} is defined as the fundamental region $R$ together with the points with nontrivial stabilizer.
\smallskip

Let $\Gamma \triangleleft \Delta$ be a self-duality pairing where $\Gamma$ and $\Delta$ are finite groups of isometries of the sphere, such that $\Gamma$ is an  index 2 subgroup of $\Delta$. In order to describe each orbifold we need fundamental regions $R_1$ and $R_2$ corresponding to the groups $\Gamma$ and $\Delta$, respectively as well as their corresponding generators. We observe that since $\Delta$ is an extension of index 2,  region $R_2$ is contained in $R_1$ and it will have half of its area. 
\smallskip

In what follows we describe the fundamental regions for each antipodal self-dual pairing. We have drawn with :

- a thick line the singular elements of $\Gamma$, i.e., sets of fixed points under the action of $\Gamma$ 

- a double line the singular elements of $\Delta$

- continuous lines denote reflection lines

- dotted lines delimit the regions without considering a reflection and

- points marked with a circle denote a center of rotation by the marked angle.

\begin{tabular}{cc}
\begin{minipage}{6cm}
$\mathbf{[q]\triangleleft [2,q]}$\\
$[q]=\langle \gamma _1, \gamma _2\rangle$ is a group generated by reflections $\gamma _1$ and $\gamma _2$ through the thick lines. The fundamental region $R_1$ is a bigon formed by these lines. $[2,q]=\langle \gamma _1,\gamma _2, \gamma _3\rangle$ is generated by reflections $\gamma _1, \gamma _2$ and $\gamma _3$ and its fundamental region $R_2$ is a triangle with angles $\frac{\pi}{q},\frac{\pi}{2},\frac{\pi}{2}$.
\end{minipage}
&
\begin{minipage}{11cm}
\raisebox{-\totalheight}{\includegraphics[width=.9\textwidth]{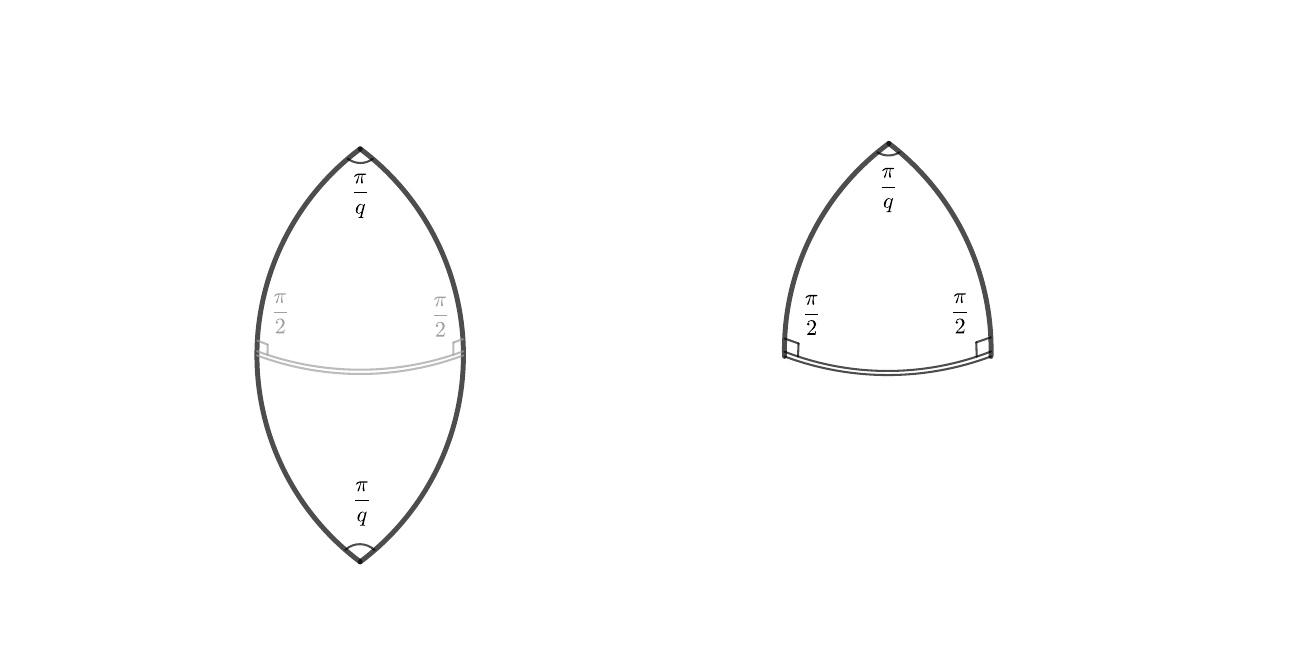}}
 \end{minipage}
\end{tabular}

 \begin{tabular}{cc}
\begin{minipage}{6cm}
$\mathbf{[q]^{+}\triangleleft [2,q^+]}$\\
The fundamental region $R_1$ is a bigon formed by lines with an angle $\frac{2\pi}{q}$ and $[q]^+=\langle \rho \rangle$, where $\rho$ is a rotation for $\frac{2\pi}{q}$. The extension $[2,q^+]=\langle \rho , \gamma \rangle$ , being $\gamma$ the reflection through the equator.
\end{minipage}
&
\begin{minipage}{11cm}
\raisebox{-\totalheight}{\includegraphics[width=.8\textwidth]{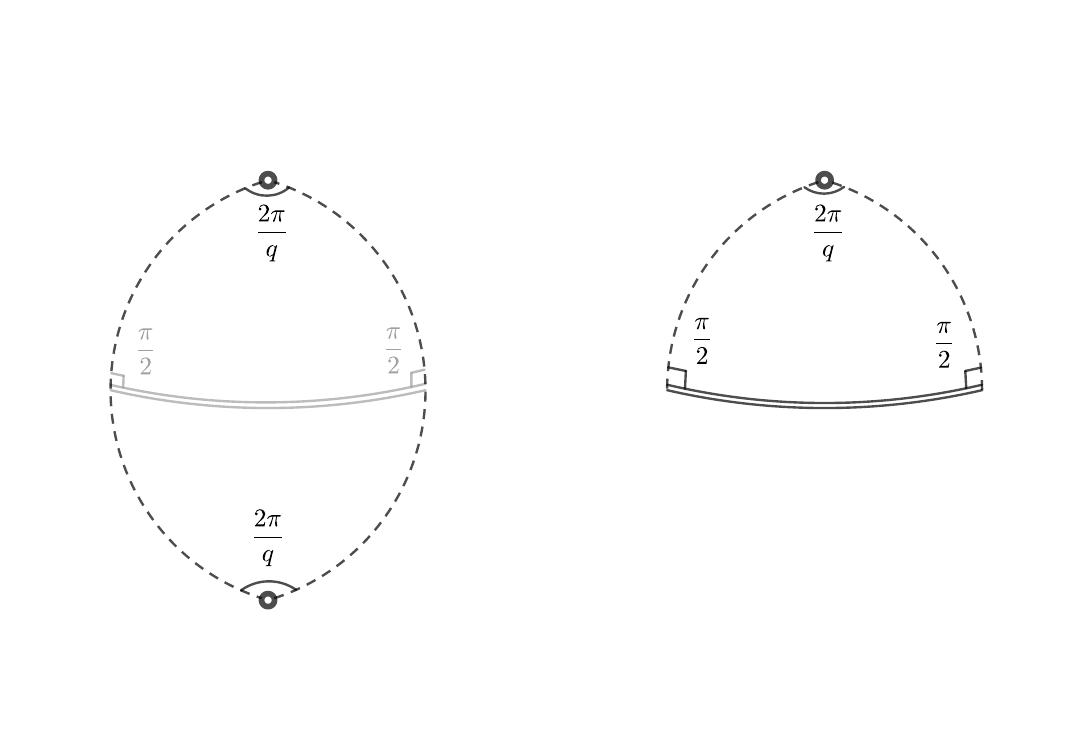}}
 \end{minipage}
\end{tabular}

\begin{tabular}{cc}
\begin{minipage}{6cm}
$\mathbf{[q]\triangleleft [2^+,2q]}$\\
The fundamental region $R_1$ is a bigon with angle $\frac{\pi}{q}$ and $[q]=\langle \gamma _1, \gamma _2\rangle$. The region $R_2$ is a triangle with angles $\frac{\pi}{q},\frac{\pi}{2},\frac{\pi}{2}$ and $[2^+,2q]=\langle \rho ,\gamma _1\rangle=\langle \rho ,\gamma _2\rangle$, where $\rho$ is a rotation by $\pi$.
\end{minipage}
&
\begin{minipage}{11cm}
\raisebox{-\totalheight}{\includegraphics[width=.8\textwidth]{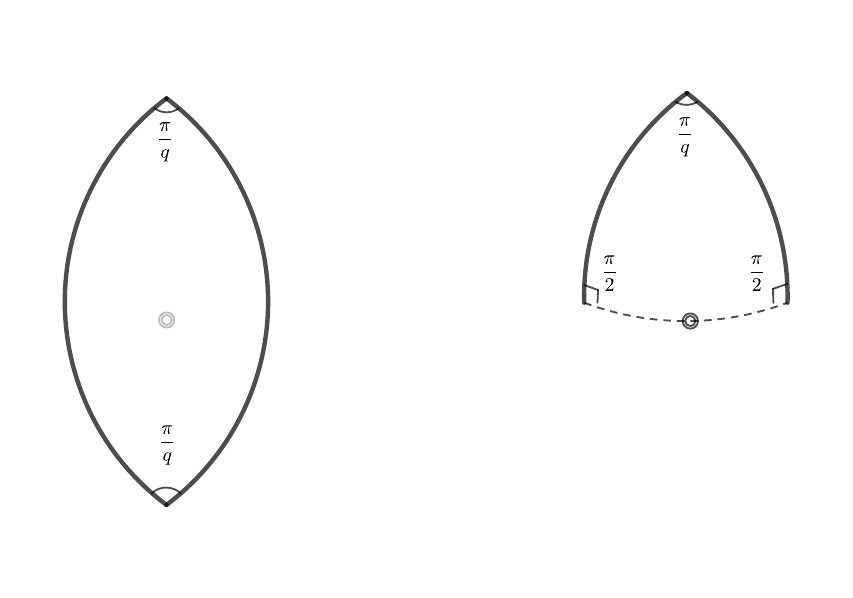}}
 \end{minipage}
\end{tabular}

\begin{tabular}{cc}
\begin{minipage}{6cm}
$\mathbf{[q]^+\triangleleft [2^+,2q^+]}$\\
The fundamental region $R_1$ consists of two triangles disposed in a way one can be obtained from the other by a rotatory reflection. This rotatory reflection is the composition of the reflection through the ``I'' mark and the rotation around the ``C''. Then $[q]^+=\langle \tau ^2 \rangle$ where $\tau$ is the rotatory reflection with order $2q$, and $[2^+,2q^+]=\langle \tau \rangle$.
\end{minipage}
&
\begin{minipage}{11cm}
\raisebox{-\totalheight}{\includegraphics[width=.8\textwidth]{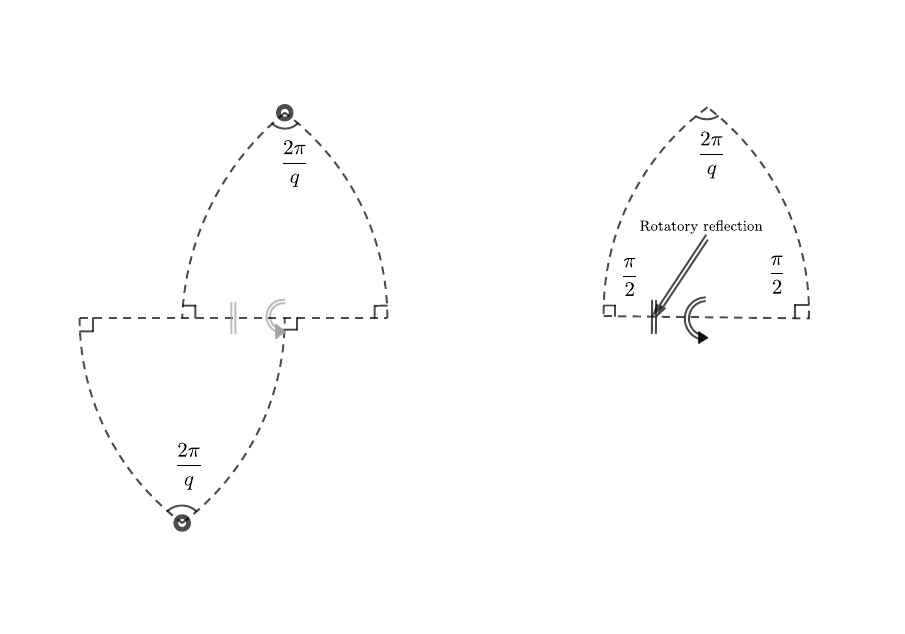}}
 \end{minipage}
\end{tabular}

\begin{tabular}{cc}
\begin{minipage}{6cm}
$\mathbf{[2,2]^+\triangleleft [2,2]}$\\
In this case $R_1$ is an equilateral triangle with internal angles of $\frac{2\pi}{3}$; $[2,2]^+=\langle \rho _{12},\rho _{23}, \rho _{31}\rangle$, $\rho _{ij}$ is a rotation by $\pi$ in the midpoint of each of the sides and $[2,2]=\langle \gamma _1,\gamma _2, \gamma _3\rangle$ where $\gamma _i$ is the reflection through a line $l_i$ and $\rho _{ij} = \gamma _j\gamma _i$.
\end{minipage}
&
\begin{minipage}{11cm}
\raisebox{-\totalheight}{\includegraphics[width=.8\textwidth]{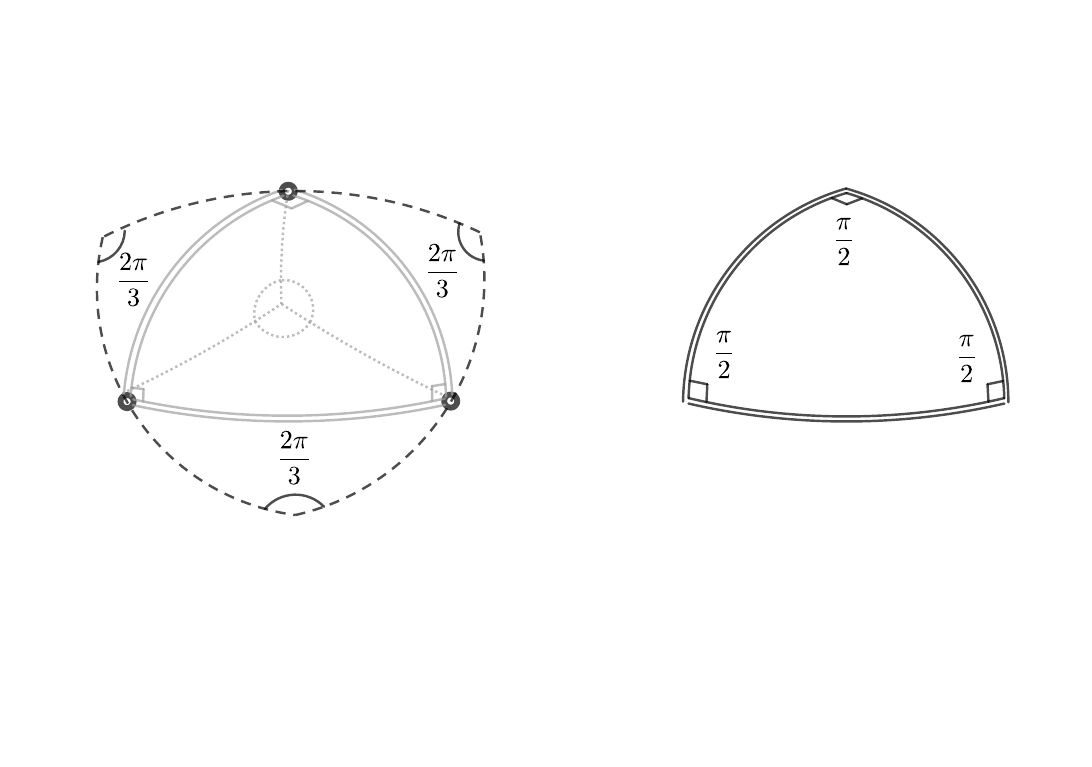}}
 \end{minipage}
\end{tabular}

 \begin{tabular}{cc}
\begin{minipage}{6cm}
$\mathbf{[2^+,4]\triangleleft [2,4]}$\\
$[2^+,4]=\langle \rho ,\gamma _1\rangle$ and $[2,4]=\langle \gamma _1,\gamma _2, \gamma _3\rangle$. Here $\gamma _i$ is a reflection through a line $l_i$ and $\rho =\gamma _3 \gamma _2$.
\end{minipage}
&
\begin{minipage}{11cm}
\raisebox{-\totalheight}{\includegraphics[width=.8\textwidth]{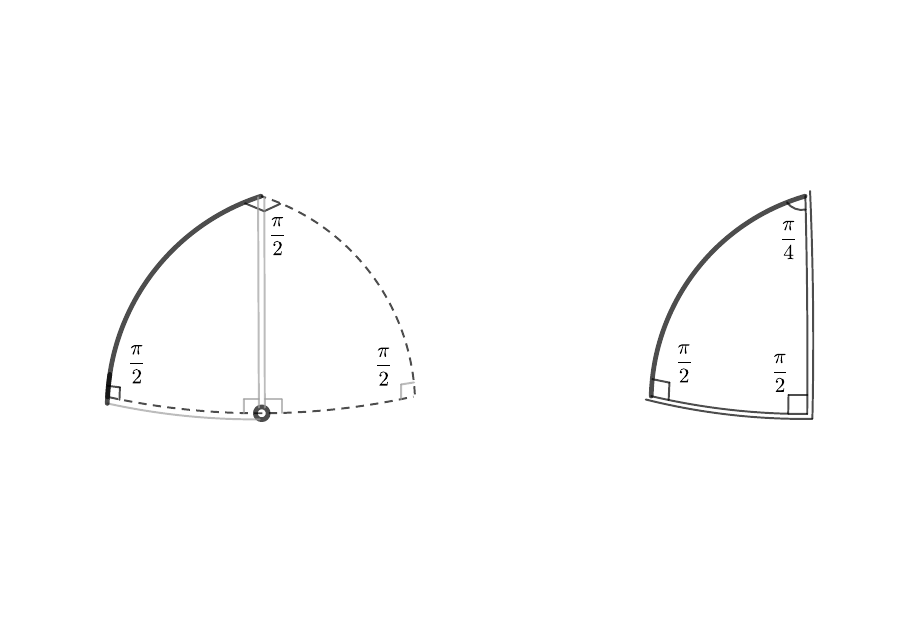}}
 \end{minipage}
\end{tabular}

 \begin{tabular}{cc}
\begin{minipage}{6cm}
$\mathbf{[2^+,4^+]\triangleleft [2,4^+]}$\\
$[2^+,4^+]=\langle \sigma \rangle$, where $\sigma$ is a rotatory reflection of order 4. $[2,4^+]=\langle \gamma ,\rho \rangle$, with $\gamma$ a reflection through the equator and $\rho$ a rotation by $\frac{\pi}{2}$ in the north pole. Note that $\sigma =  \rho \gamma$. 
\end{minipage}
&
\begin{minipage}{11cm}
\raisebox{-\totalheight}{\includegraphics[width=.8\textwidth]{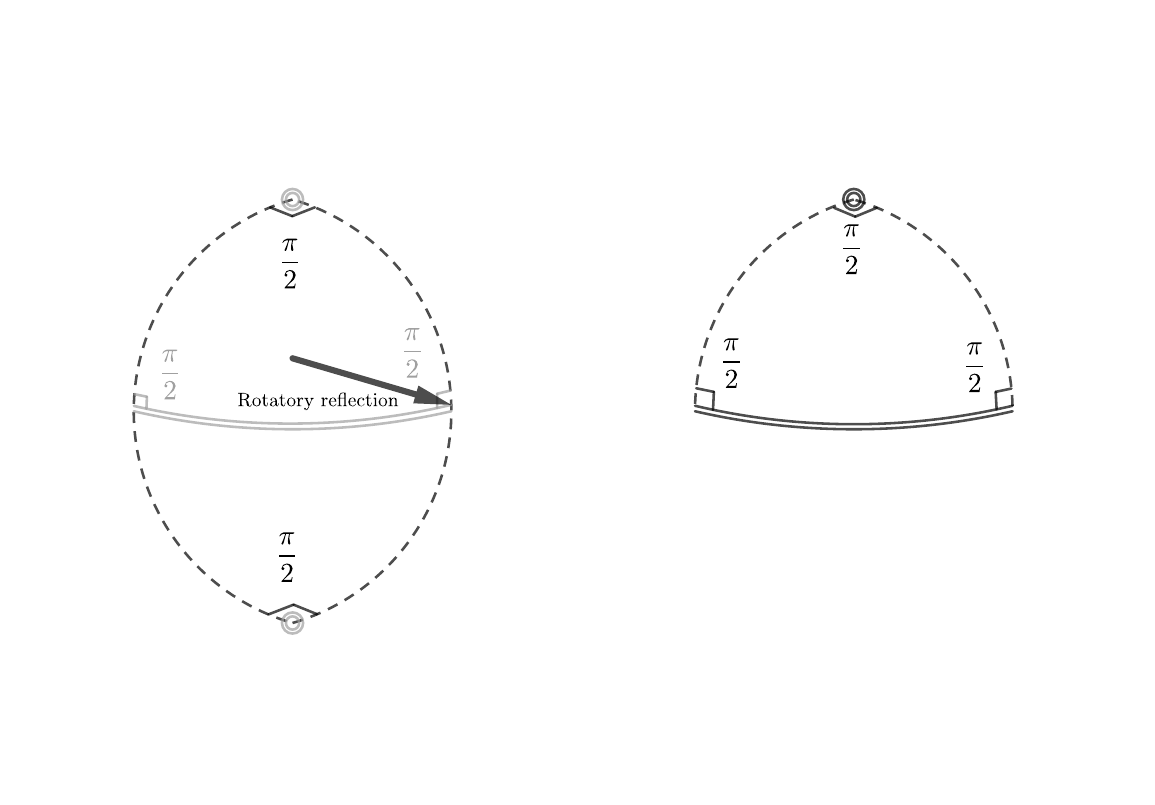}}
 \end{minipage}
\end{tabular}

 \begin{tabular}{cc}
\begin{minipage}{6cm}
$\mathbf{[1]\triangleleft [2,2^+]}$\\
In this case $R_1$ is one hemisphere and $[1]=\langle \gamma \rangle$, where $\gamma$ is the reflection through the equator, and $[2,2^+]=\langle \gamma ,\rho \rangle$, where $\rho$ is the rotation by $\pi$ through the marked point (north or south pole).
\end{minipage}
&
\begin{minipage}{11cm}
\raisebox{-\totalheight}{\includegraphics[width=.8\textwidth]{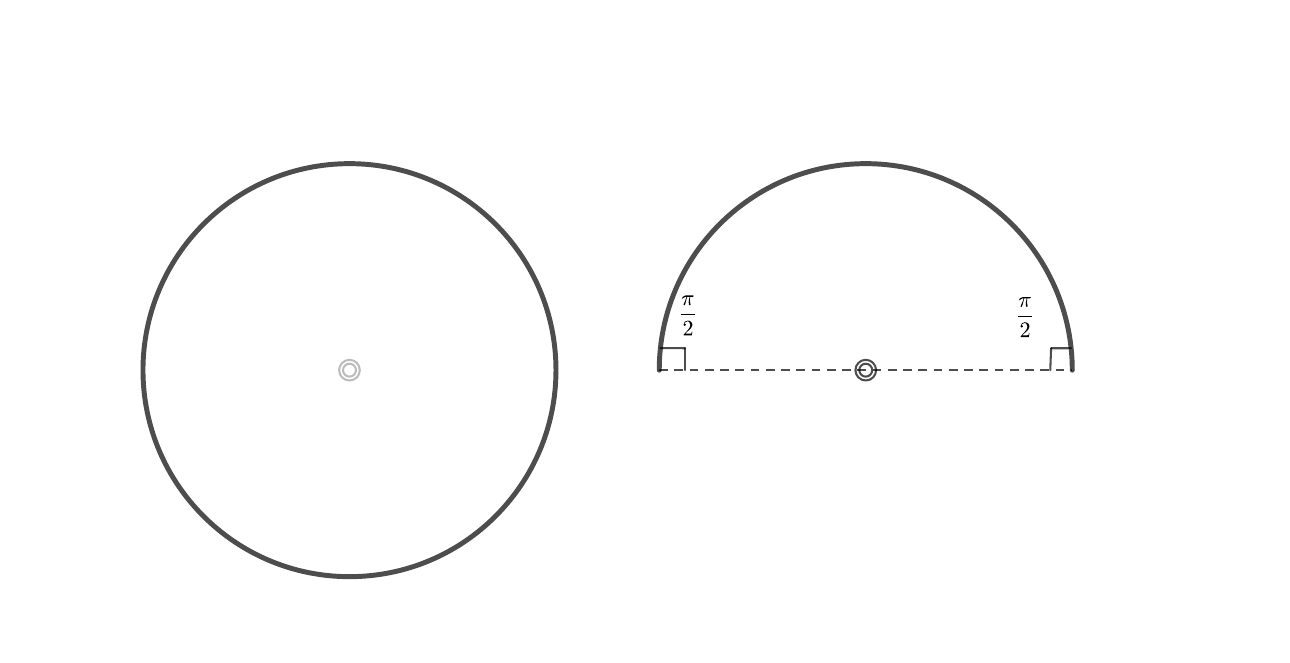}}
 \end{minipage}
\end{tabular}

 \begin{tabular}{cc}
\begin{minipage}{6cm}
$\mathbf{[3,3]\triangleleft [3,4]}$\\
$[3,3]=\langle \gamma _1,\gamma _2,\gamma _3\rangle$ and $[3,4]=\langle \gamma _1,\gamma _2,\gamma _4\rangle$. Here $\gamma _4$ is the reflection through the bisector or the right angle in the first triangle.
\end{minipage}
&
\begin{minipage}{11cm}
\raisebox{-\totalheight}{\includegraphics[width=.8\textwidth]{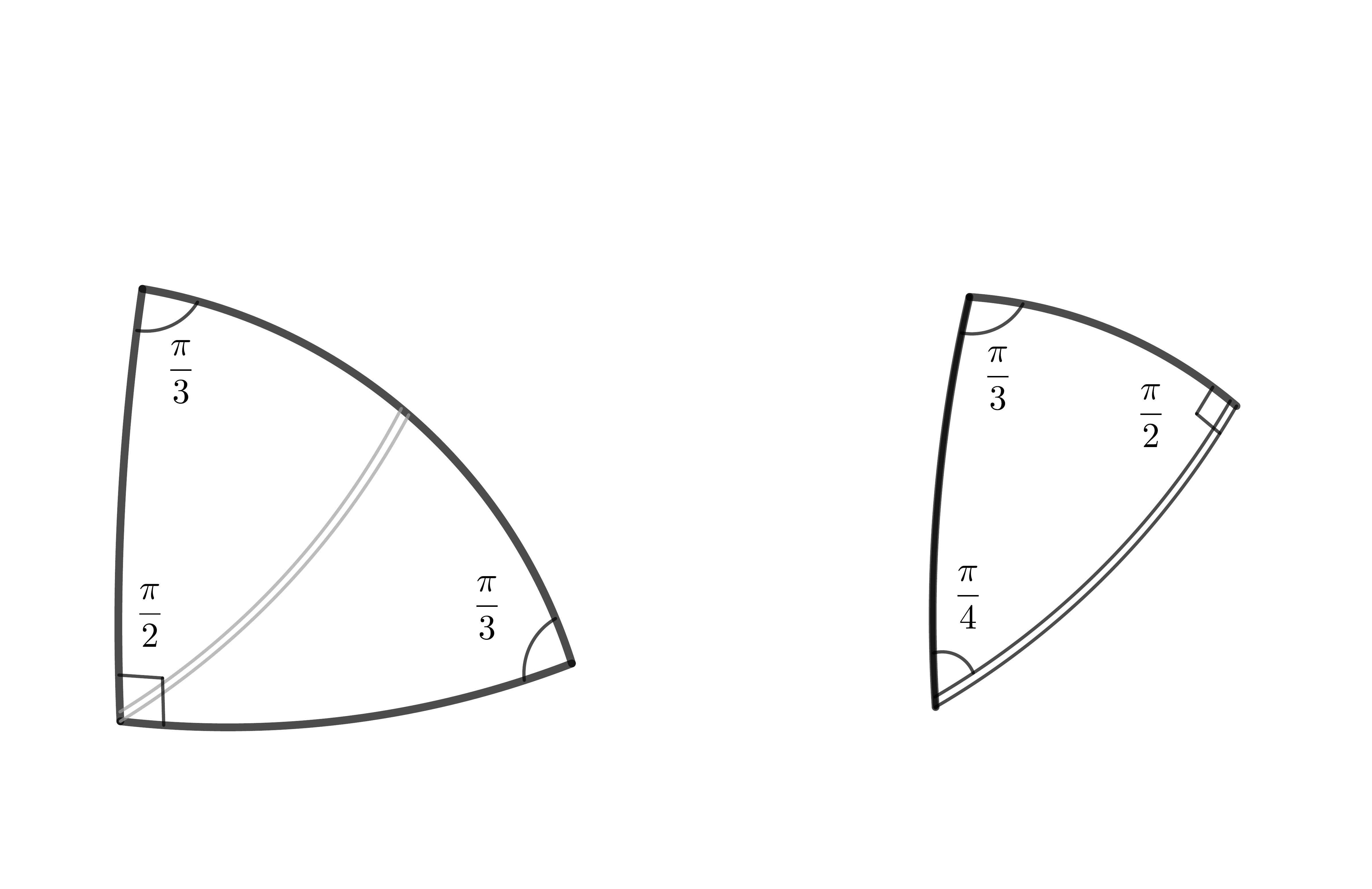}}
 \end{minipage}
\end{tabular}

 \begin{tabular}{cc}
\begin{minipage}{6cm}
$\mathbf{[3,3^+]\triangleleft [3^+,4]}$\\
$[3,3^+]=\langle \gamma _1,\rho _1\rangle$ and $[3^+,4]=\langle \gamma _2, \rho _2\rangle$. 
In this case, $R_1$ is a quadrilateral. $\gamma _1$ is a reflection through one of the thick lines (the other one can be obtained by rotation), and $\rho _1$ is the rotation around the thick circle. $R_2$ is a triangle and $\gamma _2$ is the reflection in the ``horizontal'' line and $\rho _2$ is the rotation around the circle with double line.
\end{minipage}
&
\begin{minipage}{11cm}
\raisebox{-\totalheight}{\includegraphics[width=.8\textwidth]{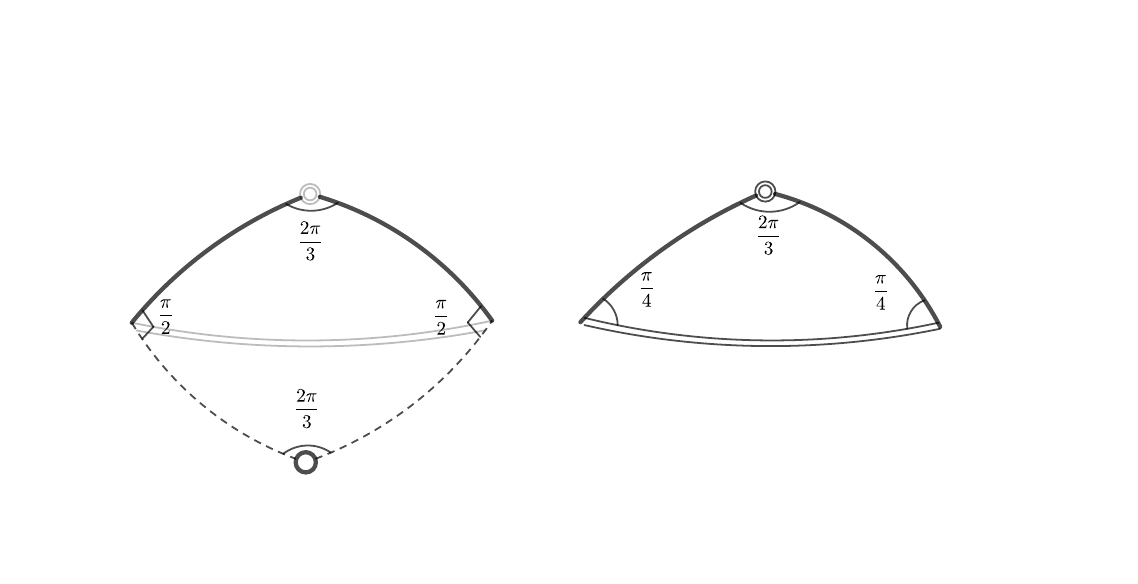}}
 \end{minipage}
\end{tabular}

\section{Constructing strongly involutive maps}\label{sec:construction}

Let us consider the orbifold of a map $G$ consisting of a fundamental region $R$ with the singular points. By construction, we have that the orbit of $R$ is a set of congruent regions covering the 2-sphere. Given a map $G$, we may consider the part of the map $G$ which is drawn in the fundamental region and does not have any other symmetry. We call this restriction the \textit{doodle} of $G$.
\smallskip

Next we describe how two involutive polyhedra are constructed from their doodle by means of the action of the duality group $\mbox{Dual}(G) \rhd \mbox{Aut}(G)$.

\subsection{The $q$-multi hyperwheel}

Let $q$ and $l$ be natural numbers. The {\em $q$-multi hyperwheel} with $l$ levels is the graph $\mathcal{O}_q^l $ consisting of the cycles:
\[(a_1^1,\ldots ,a_q^1),\cdots , (a_1^l,\ldots , a_q^l), (b_1^2,\ldots ,b_q^2), \cdots , (b_1^l, \ldots, b_q^l) \text{ and } (a_1^1,b_1^1, a_2^1, b_2^1,\ldots ,a_q^1,b_n^1)\]
as well as a vertex $c=b_i^{l+1}$ called the \textit{cusp} and the edges $a_i^j a_i^{j+1}$ and $b_i^jb_i^{j+1}$ for $j=1,\ldots ,l$; see Figure \ref{O43}.

\begin{figure}[h]
\hspace{3cm}
\includegraphics[scale=0.11]{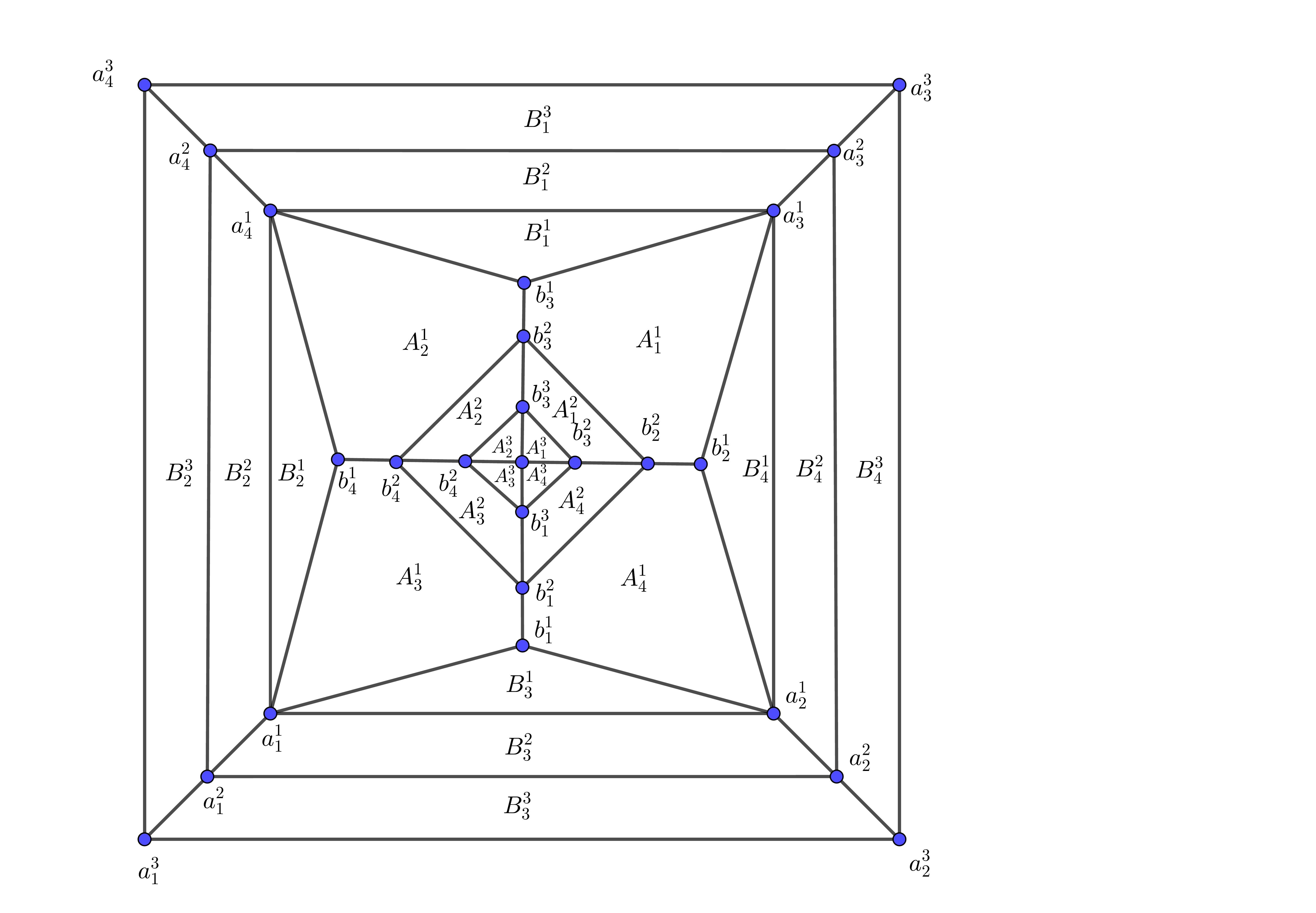}
\caption{The 4-multi hyperwheel with 3 levels $\mathcal{O}_4^3$.}
\label{O43}
\end{figure}

If $q=2k\geq 4$ then $\mathcal{O}_q^l$ is a strongly involutive self-dual polyhedron. Indeed, let us check this for $\mathcal{O}_q^l$ where the strong involution $\alpha$ is given as follows:
$$\begin{array}{lll}
\alpha (a_i^l)&=A_i^l &=(b_{i+k}^1b_{i+k-1}^1c),\\
\alpha (a_i^j)&=A_i^j &= (b_{i+k}^jb_{i+k-1}^jb_{i+k-1}^{j+1}b_{i+k}^{j+1}), \ \text{ for each } j=2,\ldots ,l-1,\\
\alpha (a_i^1)&=A_i^1 &= (b_{i+k}^2b_{i+k-1}^2b_{i+k-1}^1a_{i+k}^1b_{i+k}^1),\\
\alpha (b_i^1)&=B_i^1 &= (a_{i+k}^1b_{i+k}^1b_{i+k-1}^1),\\
\alpha (b_i^j)&=B_i^j &= (a_{i+k}^{j-1}a_{i+k-1}^ja_{i+k-1}^{j}a_{i+k-1}^{j-1}), \ \text{ for each } j=2,\ldots ,l,\\
\alpha (c)&=C &=(a_1^l\cdots a_q^l).\\
\end{array}$$

Note that $\alpha$, as duality isomorphism, associates the vertex set of the dual face to each point $p$, while, as an isometry of the sphere, it associates the antipodal point $-p$ to each point $p$.
\smallskip

We may now construct the 4-multi hyperwheel $G=\mathcal{O}_4^1$ which is just the 4-hyperwheel as illustrated in Figure \ref{fig1} (right) by means of the action of its antipodal pairing. 
Figure \ref{O4} gives the fundamental region $R$ and its doodle. This doodle can generate a family of self-dual polyhedra whose pairing is $[q] \triangleleft [2,q]$.

\begin{figure}[h]
\vspace{-1cm}
\centering
\includegraphics[scale=0.7]{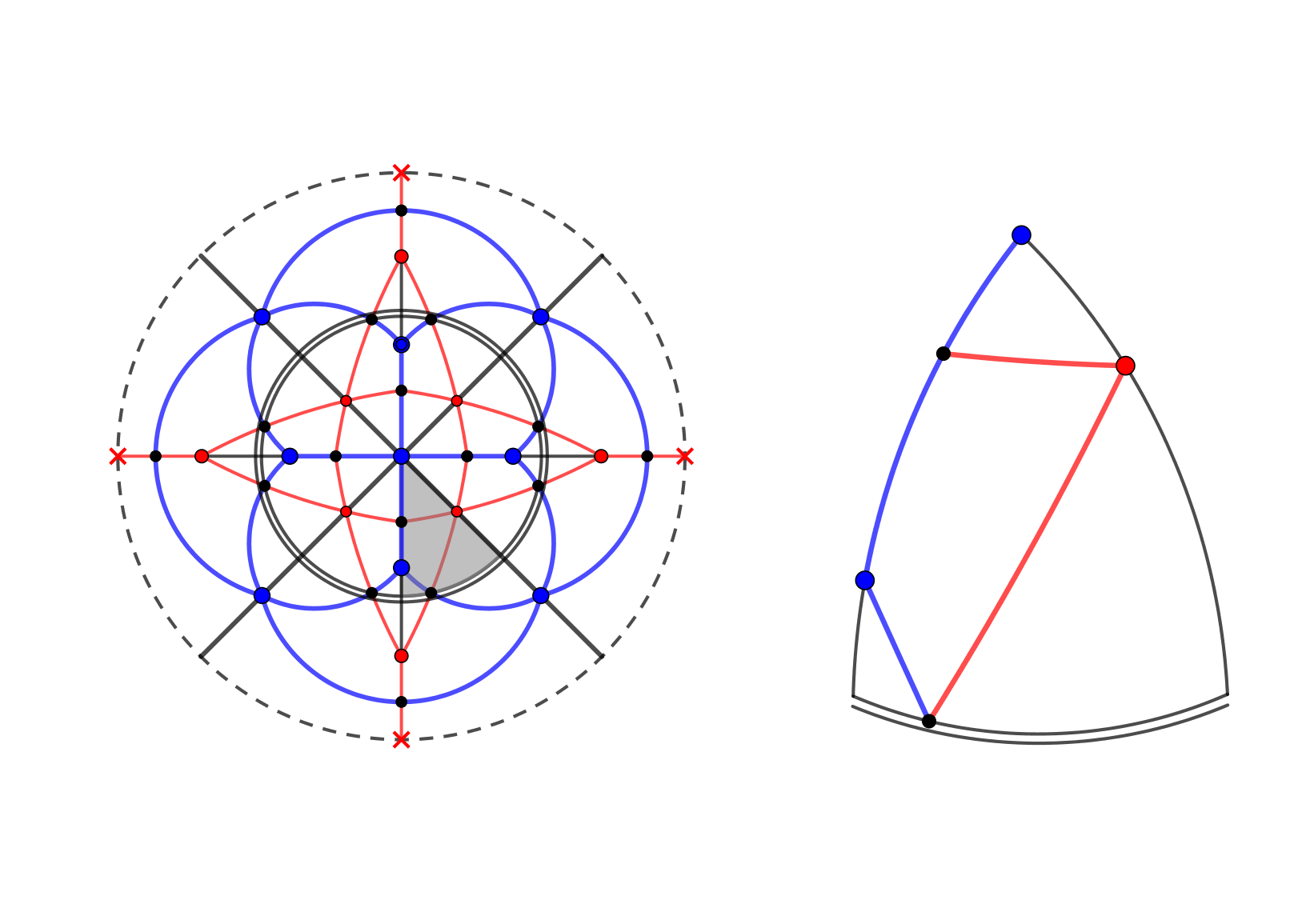}
\caption{(Left) A fundamental region (shaded) of $[4]\triangleleft [2,4]$. (Right) The doodle of the 4-hyperwheel.}
\label{O4}
\end{figure}

The external circle with dotted line represents the north pole, the circle with double line represents the equator and the central point is the south pole. The graphs $G$ and $G^{*}$ appear in blue and red, respectively and there is a black vertex in each intersection of an edge in $G$ with an edge in $G^{*}$. The join of these graphs (including the black vertices) is the graph of squares $G_{\square}$ whose all faces are quadrilaterals of the form $(vafb)$ for some $v\in V(G), f\in V(G^{*})$ and $a,b\in E(G)=E(G^{*})$.

\begin{rem}
Let $\phi$ be any element in $\mbox{Dual}(G)$. We have that $\phi$ preserves squares. Furthermore, $\phi \in \mbox{Aut}(G)$ if and only if $\phi$ is \textit{color preserving}; $\phi \in \mbox{Iso}(G)$ if and only if $\phi$ is \textit{color reversing}. In any case, black vertices are mapped to black vertices.
\end{rem}

We suppose that $\phi$ is the reflection in a spherical line $l$ which intersects the square $(vafb)$. If $\phi \in \mbox{Iso}(G)$ then $l$ must pass through $a$ and $b$ and must interchange $v$ and $f$. In this case, the reflection in the equator is a duality isomorphism since it interchanges vertices blue and red. Each automorphism can be obtained as a composition of reflections in four lines through the poles. The reflections through these lines generates the group $[4]$ and each one of them works as an automorphism, so we can write $\mbox{Aut}(G)=[4]$. The group $\mbox{Dual}(G)$ is $[2,4]$ and it can be obtained by adding the reflection through the equator. In this way, $\mbox{Dual}(G)$ is generated by the reflections in three planes $H_1,H_2$ and $H_3$. Planes $H_1$ and $H_2$ intersect with angle $\frac{\pi}{4}$ and we can assume they are vertical planes through the poles, while $H_3$ is the equatorial plane and it forms angles of $\frac{\pi}{2}$ with each $H_1,H_2$.  The fundamental region is the spherical triangle with angles $\frac{\pi}{2},\frac{\pi}{2},\frac{\pi}{4}$. The self-dual pairing for this graph is $[4]\triangleleft [2,4]$ and it is determined by this triangle with three isometries: two reflections in $\mbox{Aut}(G)$ (black lines) and one in $\mbox{Iso}(G)$ (double line).

\begin{figure}[h]
\centering
\includegraphics[scale=0.8]{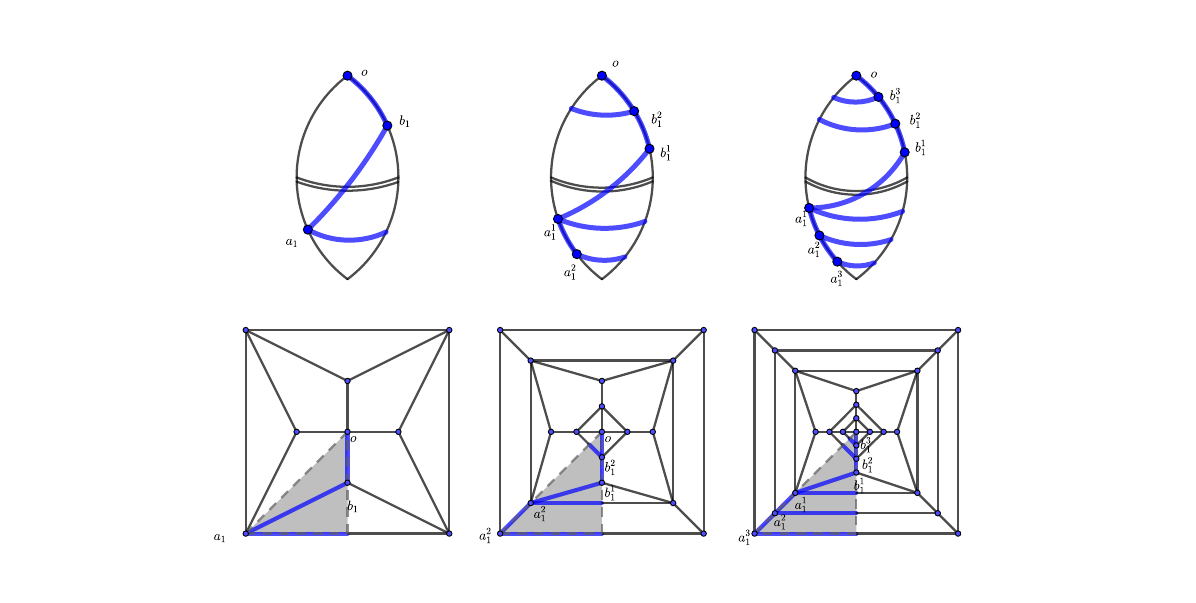}
\caption{The doodle on a fundamental region $R$ of $\mathcal{O}_q^l$ and the resulting graphs for $q=4$.}
\label{doodleO_ql}
\end{figure}

The above construction can be extended for each even $q$ and any $l$: For $q=4$ and $l=1$ we have the doodle of  $G=\mathcal{O}_4^1$ (Figure \ref{O4}). Then for $l=2$ and $3$ we have the doodles and the graphs in Figure \ref{doodleO_ql}. Here the doodle is drawn on the region between two planes forming an angle $\frac{\pi}{q}$: the fundamental region of the group $[q]$.
\vspace{0.2cm}

In Figure \ref{O_43}, we consider the embedding of both: the graph in blue and its dual graph in red. Here the reflection through the double line interchanges colors blue and red since it is a duality isomorphism of $[q]\triangleleft [2,q]$. If we let the group $[q]$ act on the sphere, we will have a generalization of the $q$-multi hyperwheel. Figure \ref{O_43} illustrates the case $q=4$ and $l=3$.

\begin{figure}[h]
\begin{center}
\includegraphics[scale=0.32]{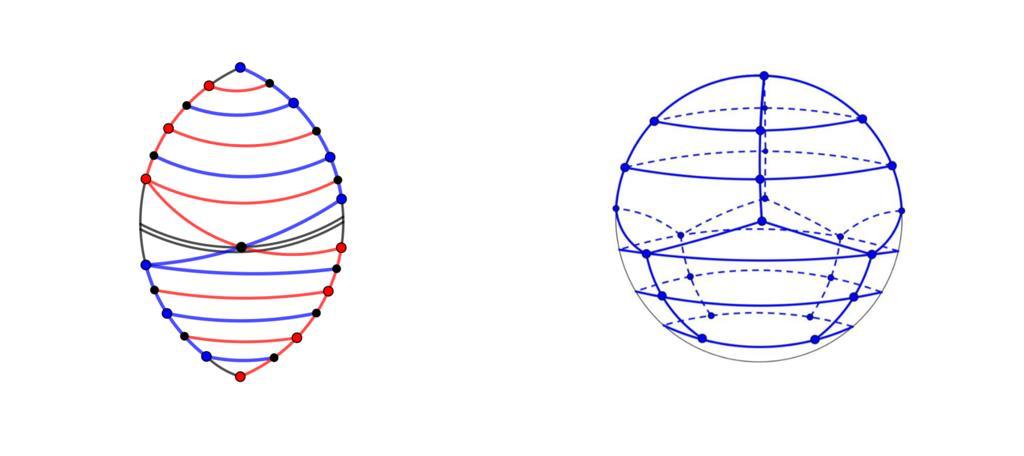}
\end{center}
\caption{(Left) The doodle and fundamental region $R$ of $\mathcal{O}_4^3$ (Right) Resulting graph embedded on $\mathbb{S}^2$.}
\label{O_43}
\end{figure}

\subsection{The $q$-multi wheel} 

Let $q$ and $l$ be natural numbers. The {\em $q$-multi wheel} with $l$ levels is the graph $\mathcal{P}_q^l$ consisting of the $l$ cycles 
$$(a_1^1,\ldots ,a_q^1), \cdots , (a_1^l,\ldots ,a_q^l),$$ a central vertex $c= a_i^0$ called the \textit{cusp} and edges $a_i^ja_i^{j+1}$. Figure \ref{P3} illustrates  $\mathcal{P}_5^3$.

\begin{figure}[h]
\centering
\includegraphics[scale=0.8]{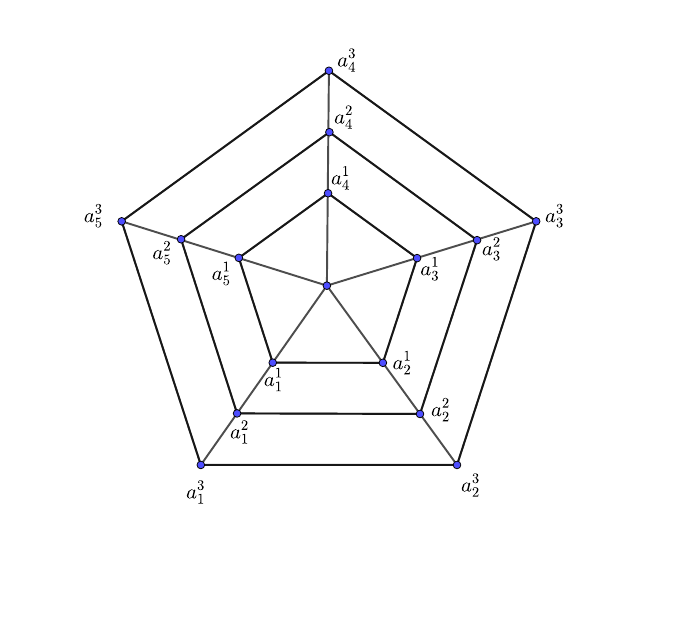}
\caption{The 5-multi wheel with 3 levels.}
\label{P_53}
\end{figure}

We notice that $\mathcal{P}_q^1$ is actually the $q$-wheel (a graph consisting of a $q$-cycle with a center joined to each vertex of the cycle).
\smallskip

If $q=2k+1\geq 3$ then $\mathcal{P}_q^l$ is a strongly involutive self-dual polyhedron. Indeed, let us check this for $\mathcal{P}_q^l$ where the strong involution $\alpha$ is given as follows:
$$\begin{array}{ll}
\alpha (a_i^j)&=(a_{i+k}^{l-j}a_{i+k}^{l-j+1}a_{i-k}^{l-j}a_{i-k}^{l-j+1}), \text{ for each } i=1,\ldots ,q \text{ and } j=1,\ldots ,l-1,\\
\alpha (a_i^l)&=(a_{i+k}^{1}a_{i-k}^{1}c),\\
\alpha (c) &= (a_{1}^{l}\cdots a_q^l).
\end{array}$$

We may now construct the 3-multi wheel $G=\mathcal{P}_3^1$ by means of the action of its antipodal pairing. Figure \ref{O4} gives the fundamental region $R$ and its doodle. This doodle can be naturally extended to generate a family of self-dual polyhedra whose pairing is $[q] \triangleleft [2,q]$. 
\smallskip

Let us consider an isometric embedding of the graph $G=\mathcal{P}_3$ given in Figure \ref{P3}. 

\begin{figure}[h]
\centering
\includegraphics[scale=0.81]{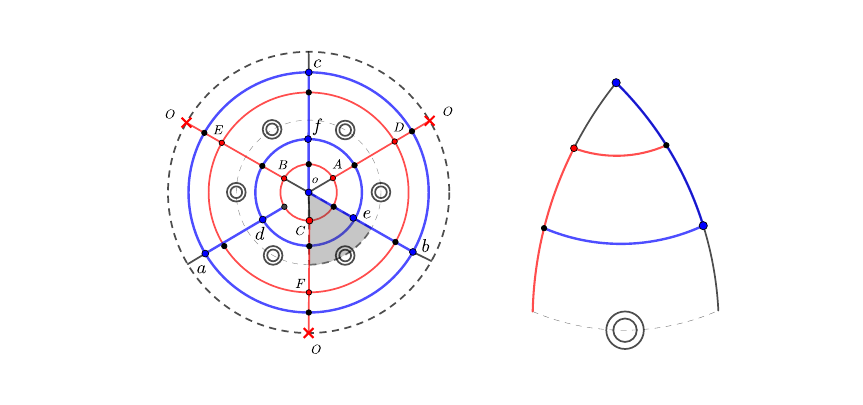}
\caption{(Left) A fundamental region of $[3]\triangleleft [2^+,6]$. (Right) The doodle of $\mathcal{P}_3^1$.}
\label{P3}
\end{figure}

In this embedding, the vertices $a,b$ and $c$ are in a circle while $d,e$ and $f$ are in another circle. These circles are concentric to the center $o$, which we can assume to be the south pole. The red vertices correspond to the dual graph and they are also in concentric circles, in this case $D,E$ and $F$ are connected with the central vertex $O$ located in the north pole (dotted external line in the figure).
\smallskip

We can observe that the reflection on the equator does not work as duality but the rotation by $\pi$ around the circle in double lines in the shaded region is a duality isomorphism. This rotation is color reversing: $e$ goes to $F$, the red point $C$ goes to $b$ and the central vertex $o$ goes to $O$. If we call this rotation $\rho$, then the group $\mbox{Dual}(\mathcal{P}_3^1)$ is generated by $\gamma _1, \gamma _2, \rho$ and in Coxeter notation it corresponds to $[2^{+},6]$. The duality pairing is $[3]\triangleleft [2^{+},6]$. 
\vspace{0.2cm}

The above construction can be extended for each odd $q$ and any $l$. We can start by restricting the graph $P_5^1$, which is the 5-wheel, to a fundamental region and then we can add more levels in order to obtain a simple, planar and 3-connected graph. See Figure \ref{doodleP_ql}.

\begin{figure}[h]
\begin{center}
\includegraphics[scale=0.62]{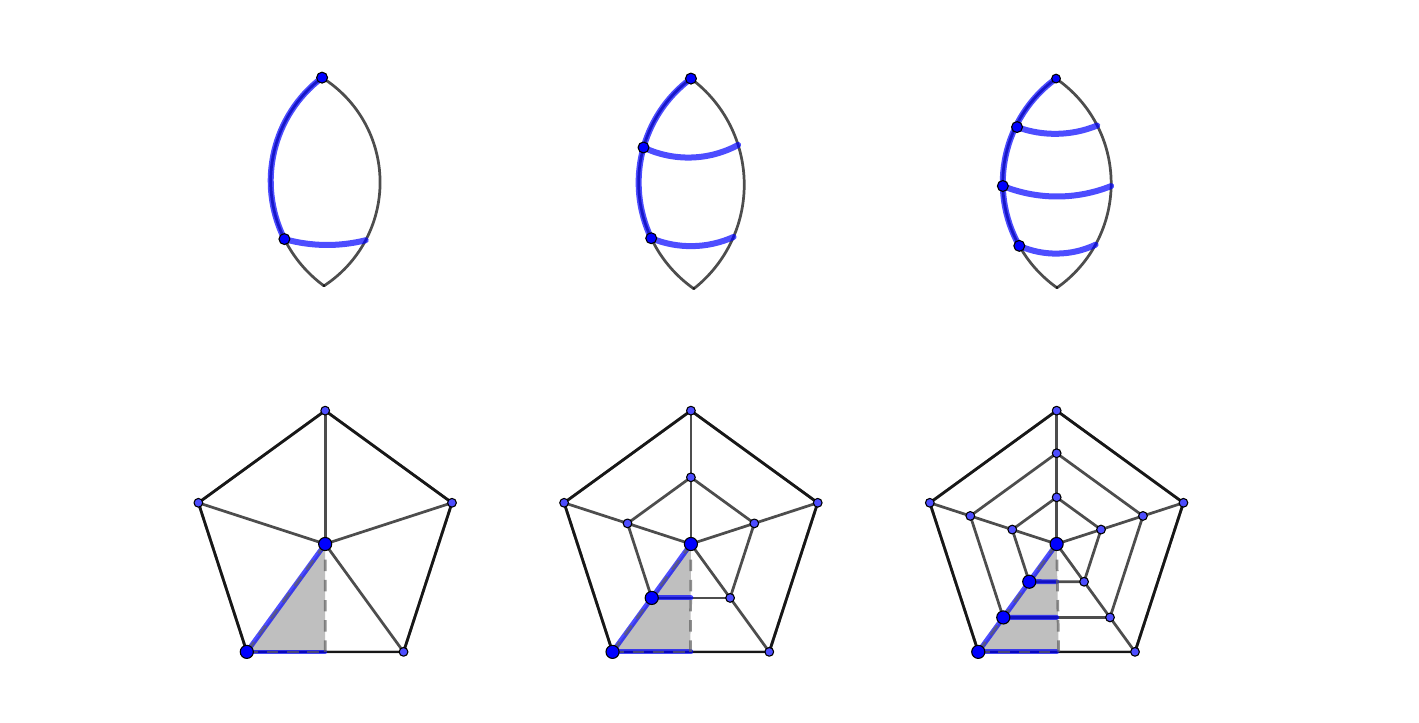}
\end{center}
\caption{The doodle on a fundamental region $R$ of $\mathcal{P}_q^l$ and the resulting graphs for $q=5$.}
\label{doodleP_ql}
\end{figure}

In Figure \ref{P_53a} we have a fundamental region $R$ for the group $[q]$. It is between two vertical planes forming an angle of $\frac{\pi}{q}$. The circle in the center of the region denotes a rotation by $\pi$. This rotation sends the blue lines in the doodle to the red lines in the doodle and viceversa since it is a duality isomorphism in the pairing $[q]\triangleleft [2^+ ,2q]$. If we let $[q]$ act on $\mathbb{S}^2$ we may obtain $\mathcal{P}_q^l$. 

\begin{figure}[h]
\begin{center}
\includegraphics[scale=0.37]{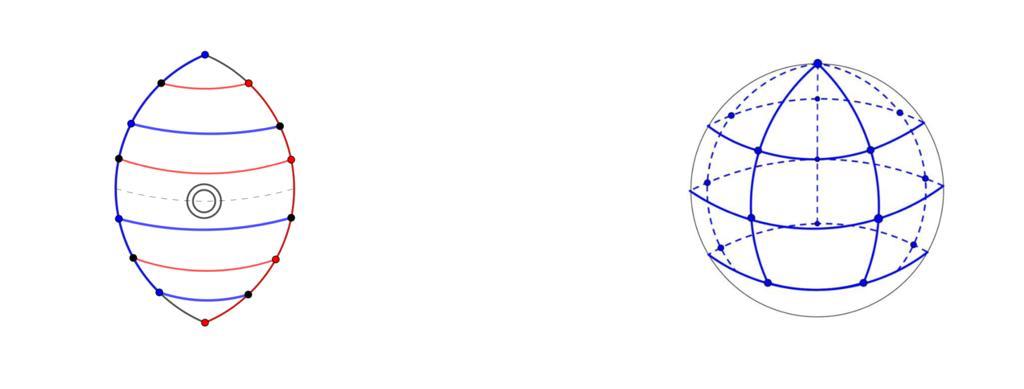}
\end{center}
\caption{(Left) The doodle in a fundamental region $R$ of $\mathcal{P}_5^3$. (Right) Resulting graph embedded on $\mathbb{S}^2$.}
\label{P_53a}
\end{figure}
\section{Concluding Remarks}

In \cite[Theorem 6] {BMPRA}, the authors showed that if $G$ is a strongly involutive polyhedron then there always exists
an edge $e \in E(G)$ such that $G'=G/\{e\} \setminus \{\tau (e)\}$ is also an strongly involutive polyhedron where $\tau$ is the strong involution and $G \setminus \{f\}$ (resp. $G/\{f \}$) denotes the deletion (resp.\ contraction) of an edge $f$ in $G$. It can be checked that the inverse of the delete-contraction operation corresponds to the diagonalization faces of the graph and its dual simultaneously. The latter can be settled as an {\em add-expansion} operation from $G'$ to obtain $G$.
\smallskip

The above implies that all strongly involutive polyhedra can be reduced to a wheel (with an odd number of vertices in the main cycle) by a finite sequence of {\em delete-contraction} operations (applied simultaneously each time), see \cite[Corollary 7] {BMPRA}. We end with the following

\begin{question} Let $G$ be a strongly involutive polyhedron and let $G'$ be the strongly involutive polyhedron obtained from $G$ by applying a delete-contraction operation. Can the antipodal pairing of $G$ be determined from the antipodal pairing of $G'$? What about the corresponding orbifolds and doodles? 
\end{question}

\end{document}